\documentclass[final,onefignum,onetabnum]{siamart171218}

\usepackage{amsmath}
\usepackage{amssymb}
\usepackage{mathtools}
\usepackage{enumitem}
\usepackage{etoolbox}
\usepackage{ulem}
\usepackage{subcaption}
\usepackage[colorinlistoftodos]{todonotes}
\reversemarginpar

\usepackage{MnSymbol,graphicx}

\newtheorem{thm}{Theorem}[section]

\newtheorem{prop}{Proposition}[section]
\newtheorem{defn}{Definition}[section]

\newtheorem{assumption}{Assumptions}[section]
\newtheorem{rem}{Remark}[section]

\usepackage{lipsum}
\usepackage{amsfonts}
\usepackage{graphicx}
\usepackage{epstopdf}
\usepackage{algorithmic}
\ifpdf
  \DeclareGraphicsExtensions{.eps,.pdf,.png,.jpg}
\else
  \DeclareGraphicsExtensions{.eps}
\fi

\newsiamremark{remark}{Remark}
\newsiamremark{hypothesis}{Hypothesis}
\crefname{hypothesis}{Hypothesis}{Hypotheses}
\newsiamthm{claim}{Claim}

\headers{Regularization on compact sets}{Barbara Palumbo, Paolo Massa, and Federico Benvenuto}

\title{Convergence rates for Tikhonov regularization on compact sets: application to neural networks}

\author{Barbara Palumbo\thanks{MIDA, Dipartimento di Matematica, Università di Genova, via Dodecaneso 35, I-16146 Genova,
Italy
(\email{barbara.palumbo@edu.unige.it}, \email{federico.benvenuto@unige.it})}
\and Paolo Massa\thanks{University of Applied Sciences and Arts Northwest Switzerland (FHNW), School of Computer Science, Bahnhofstrasse 6, Windisch 5210, Switzerland 
  (\email{paolo.massa@fhnw.ch}).}
\and Federico Benvenuto\footnotemark[1]}

\usepackage{amsopn}

\makeatletter
\newcommand*{\addFileDependency}[1]{
  \typeout{(#1)}
  \@addtofilelist{#1}
  \IfFileExists{#1}{}{\typeout{No file #1.}}
}
\makeatother

\newcommand*{\myexternaldocument}[1]{%
    \externaldocument{#1}%
    \addFileDependency{#1.tex}%
    \addFileDependency{#1.aux}%
}

\ifpdf
\hypersetup{
  pdftitle={Tikhonov and neural networks},
  pdfauthor={Palumbo B.}
}
\fi

\myexternaldocument{ex_supplement}

\begin{document}

\maketitle

\begin{abstract}
In this work, we consider ill-posed inverse problems in which the forward operator is continuous and weakly closed, and the sought solution belongs to a weakly closed constraint set. 
We propose a regularization method
based on minimizing the Tikhonov functional on a sequence of compact sets which is dense in the intersection between the domain of the forward operator and the constraint set.
The index of the compact sets can be interpreted as an additional regularization parameter.
We prove that the proposed method is a regularization, achieving the same convergence rates as classical Tikhonov regularization and attaining the optimal convergence rate when the forward operator is linear.
Moreover, we show that our methodology applies to the case where the constrained solution space is parametrized by means of  neural networks (NNs), and the constraint is obtained by composing the last layer of the NN with a suitable activation function.
In this case the dense compact sets are defined by taking a family of bounded weight NNs with increasing weight bound.
Finally, we present some numerical experiments in the case of Computerized Tomography to compare the theoretical behavior of the reconstruction error with that obtained in a finite dimensional and non-asymptotic setting.
The numerical tests also show that our NN-based regularization method is able to provide piece-wise constant solutions and to preserve the sharpness of edges, thus achieving lower reconstruction errors compared to the classical Tikhonov approach for the same level of noise in the data.
\end{abstract}



\section{Introduction}
Regularization theory \cite{engl1996regularization} provides a framework to overcome the ill-posedness of an inverse problem by introducing a family of operators that approximate the inverse of the forward operator in a stable way. 
Each operator in the family is indexed by a \textit{regularization parameter}, which controls the trade-off between stability and accuracy. 
The most widely adopted regularization method is the Tikhonov strategy \cite{tikhonov1977solutions}, which defines the regularized solution as the minimizer of an objective functional. 
Depending on the specific problem, the Tikhonov method can be adapted to different noise models or incorporate additional prior knowledge.
These include different data misfits, such as the $L^1$ norm or the Kullback–Leibler divergence, and various penalty functionals.
A central question in this context is the study of \textit{convergence rates}, which quantify the speed at which the regularized solution converges to the exact solution as a function of the noise level, under smoothness assumptions on the true solution. 
Convergence rates for the Tikhonov method and its extensions have been thoroughly established; see, e.g., \cite{Hohage_2016,hohage2014convergence,resmerita2007joint,mair1996statistical,benning2018modern, bissantz2007convergence, kaltenbacher2024convergence, ramlau2010convergence, kindermann2016convex, kindermann2011convergence, ramlau2006tikhonov, jin2009elastic}.
Another approach to define a family of regularization operators is to restrict the domain of the forward operator, thereby limiting the space of admissible solutions. A notable example of this approach is the Ritz regularization method which defines a dense sequence of finite-dimensional vector subspaces, and finds a Tikhonov-regularized solution within each subspace. 
Convergence rates for this method are the same as those of classical Tikhonov regularization, provided that the approximating subspaces become dense at a sufficiently fast rate \cite{groetsch1984theory}. An alternative strategy to obtain stable solutions to ill-posed inverse problems is to restrict the domain of the forward operator to a compact subset of the solution space. This approach is known as the selection method \cite{tikhonov1977solutions}. While this method do not define a regularization in the classical sense, since it does not introduce a family of approximating operators, it still ensures stability under certain conditions. 
In particular, if the forward operator is injective, restricting it to a compact domain yields a well-posed problem as the inverse operators is continuous.

In this work, we consider a possibly non-linear inverse problem with a continuous and weakly closed forward operator. Unlike in previous studies, we assume that the sought solution belongs to a weakly closed constraint set. 
To provide a family of regularized solutions in this general case, we define a dense sequence of compact sets and we determine the Tikhonov solution within each set.
This construction blends two distinct philosophies: on the one hand, it reflects the idea of the Ritz method of approximating the solution within a dense sequence of sets; on the other hand, by selecting compact sets, it aligns with the strategy of the selection method. 
Doing so, we reinterpret the selection method as a genuine regularization strategy.
The main advantage of this reinterpretation is that the regularized solution can be defined as a neural network (NN) \cite{Goodfellow2016}, since the sequence of compact sets can be realized by neural networks with increasing bounded weights and  number of parameters (e.g., the number of neurons), thanks to their universal approximation property which guarantees the ability to approximate the exact solution to any desired level of accuracy \cite{leshno1993multilayer}.
This leads to the additional benefit that in practice a regularized solution can be easily computed by optimizing the NN weights with widely-adopted gradient descent techniques \cite{kingma2017adammethodstochasticoptimization}.
Moreover, the use of NNs allows imposing constraints (e.g., non-negativity) on the inverse problems solution by applying an appropriate activation function at the last layer. 
This use of NNs substantially differs from the typical machine learning approaches adopted in the context of inverse problems. 
In this regard, machine learning has been used primarily for three purposes: (i) selecting or learning the regularization parameter, (ii) approximating the inverse operator, and (iii) designing data-adaptive regularization functionals or learned priors. 
In these settings as well, the literature offers an analysis of convergence rates; see, for example, \cite{Arridge_Maass_Öktem_Schönlieb_2019, ongie2020deep, li2020nett, afkham2021learning, duff2024regularising, guastavino2020convergence}. 
However,  our regularized solution can be viewed as a form of learned positional encoding, where spatial points are embedded via a NN, similarly to recent approaches \cite{vaswani2017attention}. 
A particularly relevant example of this technique is Neural Radiance Fields (NeRF), which was recently proposed and has been extensively applied to the problem of scene reconstruction and novel view synthesis from a limited number of viewpoints
\cite{mildenhall2021nerf, barron2021mip, wang2021neus,kerbl20233d}.

We prove that our proposed strategy is a regularization method based on two parameters: in addition to the classical Tikhonov regularization parameter, the second parameter corresponds to the index of the compact sets.
Our proofs follow the analysis of non-linear inverse problems presented in Chapter 10 of \cite{engl1996regularization}.
We obtain convergence rates for classical source conditions on solutions belonging to a weakly closed constraint set, in line with those for regularization methods under non-negativity constraints \cite{clason2019regularization},  or closed and convex constraints \cite{Chavent_1994}.
Analogously to the Ritz regularization method, our convergence rates depend on the rate at which the sequence of sets becomes dense within the solution space. 
We also prove that, when applied to linear operators, the proposed regularization method achieves the optimal convergence rate.
As far as we know, convergence rates for this method has been studied only in the case of  denoising, i.e., when the forward operator is the identity
\cite{burger2000training}.
In this case,  convergence rates have been established when the penalty term is either the Sobolev norm of the solution or the classical $\ell^2$-norm applied to the weights in the last layer.
Therefore, theoretical contributions of this paper on convergence rates can be viewed as a extension of the results obtained in \cite{burger2000training}  when the forward operator is weakly closed and the solution belongs to a weakly closed set.
Finally, we perform numerical experiments in the case of Computerized Tomography to test the proposed regularization method in a finite-dimensional and non-asymptotic setting.
Reconstruction errors are computed for both our method and the classical Tikhonov regularization, and the experimental results are compared with theoretical estimates.
We show that the NN-based regularization method, which includes the non-negativity constraint, provides lower reconstruction errors compared to the Tikhonov approach for the same noise level affecting the data.

The reminder of the paper is organized as follows. 
In section \ref{sec:main} we present the proposed approach for possibly nonlinear constrained inverse problems and prove that it is a regularization method; moreover we provide convergence rates for our technique and, in the case of linear operators, we show how analogous theoretical results, together with optimal convergence rates, can be obtained under milder assumptions.
In section \ref{sec:neural_networks} we show how to apply this method to the case of compact sets defined by NNs with bounded weights.
Section \ref{sec:numexp} is dedicated to the numerical experiments for testing theoretical results in a non asymptotic regime, while section \ref{sec:conclusion} contains our conclusions.

\section{Mathematical formulation} \label{sec:main}

In this paper, we address an ill-posed problem that is not necessarily linear. 
We consider $\mathcal X$ and $\mathcal Y$  Hilbert spaces and an operator $F \colon {\mathcal X} \to {\mathcal Y}$ with domain $D(F) \subseteq \mathcal{X}$.
Given $y \in {\mathcal Y}$, our goal is to solve
\begin{equation}
\label{eq: F(x)=y}
F(x) = y
\end{equation}
when only a noisy version of the data $y^\delta \in {\mathcal Y}$ such that $\|y^\delta-y\|\le\delta$ is available.
In addition, we impose a constraint on the solution, i.e.,  we require the solution to belong to a weakly closed subset $C\subseteq D(F)$. 
To obtain a stable approximation of the true solution from noisy data, it is necessary to regularize problem \eqref{eq: F(x)=y}.

\subsection{Definitions and assumptions}
In this section, we provide the main definitions and assumptions related to the problem under investigation.
While we consider a setting similar to that of Chapter 10 in \cite{engl1996regularization}, our presentation differs in that we study a constrained problem with a weakly closed constraint.

\begin{assumption}
The following assumptions are valid throughout the paper.
\begin{enumerate}
\item[A1.] $F$ is continuous.
\item[A2.] $F$ is weakly closed\footnote{An operator $F$ is {\it weakly closed} if, for every sequence ${x_n}$ such that $x_n \rightharpoonup x$ as $n \to \infty$ in $\mathcal{X}$ and $F(x_n) \rightharpoonup y$, it follows that $x\in D(F)$ and $F(x) = y$.}.
\item[A3.] The solution $x$ of the problem \eqref{eq: F(x)=y} belongs to a weakly closed subset $C \subseteq \mathcal{X}$, i.e., $x \in C \cap D(F)$ and $F(x)=y$. 
\end{enumerate}
\end{assumption}
We note that assumption A2 is satisfied by a continuous and compact operator $F$ with weakly closed domain $D(F)$.
Further, a sufficient condition for $C$ to be weakly closed is to be closed and convex.

The classic approach (e.g., \cite{engl1996regularization}) consists in restricting the solution space to a subset of elements that are close to a reference element $x^*\in{\mathcal X}$ encoding {\it a-priori} information.
 
\begin{defn}
Any element $x^\dagger \in \mathcal{X}$ such that 
\begin{eqnarray}
\label{def: x dagger}
\|x^\dagger - x^\ast \| = \min \{\|z - x^\ast \| \; | \; F(z)=y, ~z \in C \cap D(F)\}
\end{eqnarray}
is called $x^\ast$-minimum norm solution.
\end{defn}

The following proposition holds true.
\begin{prop}
\label{prop: esistenza di x dagger}
Under assumptions $A2$ and $A3$, there exists an $x^*$-minimum norm solution for problem \eqref{eq: F(x)=y}. 

\end{prop}

\begin{proof} 
Assumption $A3$ ensures that the set $\{ x \in C \cap D(F) : F(x)=y\}$ is not empty.
We will only consider the case where the set consists of an infinite number of elements, otherwise the proof is trivial.

We define a sequence $\{x_n\}_{n \in \mathbb N} \subseteq C  \cap D(F)$ such that $F(x_n)=y$ for every $n \in {\mathbb N}$ and 
\begin{equation}
\lim_{n \to \infty} \| x_{n} - x^* \| = \inf \{\| z - x^* \| \; | \; F(z)=y, ~z \in C \cap D(F) \} ~.
\end{equation}
By using the triangle inequality $| \|x_{n} \| - \| x^* \| | \le \| x_{n} - x^* \|$ and by observing that the sequence $\{\| x_n - x^* \|\}_{n \in \mathbb N}$ is bounded, we obtain that $\{ \| x_{n} \| \}_{n \in \mathbb N}$ is bounded.
Therefore, there exists a subsequence $\{x_{n(j)}\}_{j \in \mathbb N}$ and an element $w \in \mathcal{X}$ such that $x_{n(j)} \rightharpoonup w$ as $j \to \infty$.
To simplify the notation, we denote this subsequence with $\{x_j\}_{j \in \mathbb N } $.
Since $C$ is weakly closed (Assumption A3), we have  that $w \in C$. 
Furthermore, it follows from Assumption A2 that $w \in D(F) \cap C$ and $F(w)=y$. 
Finally, due to the weak lower-semicontinuity of the norm we obtain
\begin{equation}
\| w - x^* \| \le \liminf_{n \to \infty} \| x_{n} - x^* \| = \inf \{\| z - x^* \| \; | \; F(z)=y, ~z \in C \cap D(F) \} ~.
\end{equation}
We conclude that $w$ is $x^*$-minimum norm solution for \eqref{eq: F(x)=y}. 
\end{proof}
We note that an $x^*$-minimum norm solution for \eqref{eq: F(x)=y} is not unique in general.

\subsection{Tikhonov regularization on compact sets}

\label{sec:Tikhonov-regularization}
In this section, we propose a method to regularize problem \eqref{eq: F(x)=y} subject to a weakly closed constraint by using a sequence of compact sets.
Analogously to the regularization by projection on finite dimensional subspaces (see, e.g., section 3.3 in \cite{engl1996regularization}), the use of a sequence of compact sets stabilizes the inverse operator when it is injective, but convergence is not guaranteed. 
Therefore, it is convenient to use an $L^2$ norm regularization and determine a suitable parameter choice rule involving both the strength of $L^2$  regularization and the index of the compact sets.  

In the following we will consider a sequence of compact sets $\{\mathcal X_m\}_{m\in\mathbb{N}}$ such that 
\begin{equation}
\label{prop: proprietà}
\mathcal X_m \subseteq C \cap D(F) \subseteq \overline{\bigcup_{m=1}^\infty \mathcal{X}_m}~.
\end{equation}
\begin{defn}
\label{def:tik_sol_compact}
A regularized solution of the problem \eqref{eq: F(x)=y} in $\mathcal X_m$ is an element
\begin{equation}
\label{def: x m alpha delta}
x_{m, \alpha}^\delta \in \arg \min_{z \in \mathcal X_m} \| F(z) - y^\delta \|^2 + \alpha \| z- x^ \ast \|^2 ~,
\end{equation}
where $\alpha>0$ is the regularization parameter.
\end{defn}
The existence of such solutions is guaranteed by Weierstrass Theorem as the Tikhonov functional is continuous and $\mathcal X_m$ is compact.
However, without further assumptions on the compact set $\mathcal X_m$, the uniqueness of the solution is not guaranteed.

We now show that the solution of the problem \eqref{def: x m alpha delta} continuously depends on the data. Since this result holds for any compact set, we denote it with $\mathcal K$ instead of $\mathcal{X}_m$.
\begin{thm}
\label{thm: stabilità Tikhonov in xm}
Under hypophyses $A1$, $A2$, $A3$, let $\mathcal K \subseteq D(F)$ be a compact set and $\{y^j\}_{j\in\mathbb{N}}$ a sequence in $\mathcal{Y}$ such that $y^j \to y$.
We consider
\begin{equation}\label{eq:definition problem 1}
x_\alpha^j \in \arg\min_{z\in\mathcal{K}} \|F(z)- y^j \|^2 + \alpha \| z- x^\ast \|^2 ~.
\end{equation}
Then, there exists a subsequence of $\{x_\alpha^j\}_{j\in\mathbb{N}}$ converging to an element $w\in\mathcal{K}$ such that
\begin{equation}
\label{eq:definition problem}
w \in \arg\min_{z\in\mathcal{K}} \|F(z)- y \|^2 + \alpha \| z- x^\ast \|^2~.
\end{equation}
Furthermore, the limit of every convergent subsequence of $\{x_\alpha^j\}_{j\in\mathbb{N}}$ is a minimizer of the Tikhonov functional in \eqref{eq:definition problem}.
\end{thm}
\begin{proof}
We note that the set $\mathcal{K}$ is weakly compact and weakly closed.
By definition we have
\begin{equation}
\|F(x_\alpha^j) - y^j \|^2 + \alpha \|x_\alpha^j - x^* \|^2 \le \|F(z)- y^j \|^2 + \alpha \|z - x^* \|^2
\end{equation}
for every $z \in \mathcal K$. 
Then, by using the reverse triangle inequality, we obtain
\begin{equation} 
\left| \| F(x_{\alpha}^j) \| - \| y^j \|  \right|^2 + \alpha | \| x_{\alpha}^j\| - \|x^* \| |^2 \le \| F(z) - y^j \|^2 + \alpha \|z - x^* \|
\end{equation}
for every $z \in \mathcal K$.
Since the sequence $\{y^j\}_{j \in \mathbb N}$ converges, the sequences $\{\|F(x_{ \alpha}^j)\|\}_{j \in \mathbb N}$ and $ \{\|x_{ \alpha}^j\|\}_{j \in \mathbb N}$ are bounded.
From this result and the fact that $F$ is weakly closed, it follows that there exist a subsequence $\{x_{\alpha}^{i}\}_{i \in \mathbb N}$ and an element $w \in D(F) \cap C$ such that $x_{\alpha}^{i} \rightharpoonup w$ and $F(x_{\alpha}^{i}) \rightharpoonup F(w)$.
Moreover, $w \in \mathcal K$ as $\mathcal K$ is a weakly closed set.
Thanks to the weak lower semicontinuity of the norm we obtain
\begin{eqnarray}
\label{eq: inferiore semicontinuità}
\| w- x^* \| \le \liminf_{i \to \infty} \| x_{\alpha}^{i} - x^* \| & \text{and} & \|F(w)- y \| \le \liminf_{i \to \infty} \| F(x_{ \alpha}^{i}) - y^{i} \| ~.
\end{eqnarray}
Then, thanks to the sub-additivity of $\liminf$, for every $z \in \mathcal{K}$ we have
\begin{eqnarray*}
\| F(w) - y \|^2 + \alpha \| w - x^*\|^2 & \le & \liminf_{i \to \infty} \| F(x_{\alpha}^{i}) - y^{i} \|^2 + \alpha \|  x_{\alpha}^{i} - x^* \|^2 \\
&\le& \limsup_{i \to \infty} \| F(x_{\alpha}^{i}) - y^{i} \|^2 + \alpha \|  x_{\alpha}^{i} - x^* \|^2 \\
&\le& \lim_{i \to \infty} \| F(z) - y^{i} \|^2 + \alpha \| z- x^*\|^2 \\
&= & \| F(z) - y\|^2 + \alpha \| z- x^*\|^2 ~.
\end{eqnarray*}
By taking $z = w$, the inequalities become equalities and
\begin{equation}
\label{eq: Tikhonov convegrence}
\lim_{i \to \infty} \| F(x_{\alpha}^{i}) - y^{i} \|^2 + \alpha \|  x_{\alpha}^{i} - x^* \|^2 = \| F(w) - y \|^2 + \alpha \| w - x^* \|^2 ~.
\end{equation}
We have also shown that the solution $w$ is minimizer in $\mathcal K$ of \eqref{eq:definition problem}. 

It remains to be proven that $x_{ \alpha}^{i}$ strongly converges to $w$.
By contradiction, let us suppose that $x_{\alpha}^{i} \nrightarrow w$. 
Then, there exists a subsequence $\{x_{ \alpha}^{l}\}_{l \in \mathbb N}$ such that $x_{ \alpha}^{l} \rightharpoonup w $, $F(x_{\alpha}^{l}) \rightharpoonup F(w)$, and
\begin{align*}
\lim_{l \to \infty} \| x_{\alpha}^{l} - x^* \| = \limsup_{i \to \infty} \| x_{\alpha}^{i} - x^* \| := c ~.
\end{align*}
By using the first inequality in \eqref{eq: inferiore semicontinuità}, we have that $c > \| w - x^* \|$.
From \eqref{eq: Tikhonov convegrence} we obtain that
\begin{equation}  
\limsup_{l \to \infty} \| F(x_{\alpha}^{l}) - y^{l} \|^2 = \|F(w)- y\|^2 + \alpha (\|w- x^* \|^2 - c^2) < \|F(w) - y\|^2   ~,
\end{equation}
which contradicts the second inequality in \eqref{eq: inferiore semicontinuità}.
\end{proof}

Since the solution of problem \eqref{def: x m alpha delta} is not unique, the following result on convergence to an $x^\ast$-minimum norm solution will be provided for any convergent subsequence of $x_{m, \alpha}^\delta$.

\begin{thm}
\label{thm: convergenza to x dagger}
Under hypophyses $A1$, $A2$ and $A3$, let $x^\dagger$ be a $x^*$-minimum norm solution and $y^\delta \in \mathcal Y$ be such that $\| y^\delta - y \| \le \delta$. 
Assume that $F$ is Fréchet differentiable.
Let the parameter choice rules $\alpha: \mathbb R \to \mathbb R$ and $m: \mathbb R \to \mathbb N$ be such that
\begin{itemize}
\item[i)] $\alpha(\delta) \to 0$ and  $\delta^2 / \alpha(\delta) \to 0$ as $\delta \to 0$;
\item [ii)] $m(\delta) \to \infty$ as $\delta \to 0$; 
\item [iii)] for a sequence $\{x_{m(\delta)}\}_{m(\delta) \in \mathbb N}$, with $x_{m(\delta)}\in \mathcal X_{m(\delta)}$, the following holds true
\begin{equation*}
\label{eq: x_m - x dagger vs alpha}
\lim_{\delta \to 0} \frac{\|x^\dagger - x_{m(\delta)} \| ^2}{\alpha(\delta)} =0 ~.
\end{equation*}
\end{itemize}
If we consider a sequence $\delta_k \to 0$ as $k \to \infty$, then
\begin{equation}
    x_{m_k, \alpha_k}^{\delta_k} \in \arg \min_{z \in \mathcal X_{m_k}} \| F(z) - y^{\delta_k} \|^2 + \alpha_k \| z- x^* \|^2 ~, 
 \end{equation}
with $\| y^{\delta_k} - y \| \le \delta_k $, $\alpha_k \coloneqq \alpha(\delta_k)$ and $m_k \coloneqq m(\delta_k)$, has a convergent subsequence. Moreover, the limit of each convergent subsequence is an $x^*$-minimum norm solution.
\end{thm}
\begin{proof}
Let us consider a sequence $\delta_k \to 0$ and $\{x_{m_k}\}_{k \in \mathbb N}$ as in assumption $iii)$.
First of all, we show that $\{x_{m_k, \alpha_k}^{\delta_k}\}_{k \in \mathbb N}$ is bounded.
We have
\begin{eqnarray}
\label{eq: x_mk, alohak, deltak limitata}
\alpha_k \| x_{m_k, \alpha_k}^{\delta_k} - x^* \|^2 &\le& \| F(x_{m_k, \alpha_k}^{\delta_k}) - y^{\delta_k} \|^2 +  \alpha_k \| x_{m_k, \alpha_k}^{\delta_k} - x^* \|^2 \nonumber \\
& \le & \| F(x_{m_k}) - y^{\delta_{k}} \|^2 + \alpha_k \| x_{m_k} - x^* \|^2 \\
& \le & \| F(x_{m_k}) - y \|^2 + \delta_k^2 + 2 \delta_k \|F(x_{m_k}) - y \| \nonumber \\
& & + \alpha_k \| x_{m_k} - x^* \|^2 ~. \nonumber 
\end{eqnarray}
Since \( x_{m_k} \to x^\dagger \) and $F$ is Fréchet differentiable at $x^\dagger$, we eventually have
\begin{equation}
\| F(x_{m_k}) - F(x^\dagger) \| \le \|F'(x^\dagger) \| \| x_{m_k} - x^\dagger \| + o ( \| x_{m_k} - x^\dagger \| ) ~.
\end{equation}
Therefore, we obtain
\begin{eqnarray}
\label{eq: limite norma x_m, alpha^ dagger} 
\alpha_k \| x_{m_k, \alpha_k}^{\delta_k} - x^* \|^2  &\leq & \|F'(x^\dagger)\|^2 \| x_{m_k} - x^\dagger \|^2 + o(\| x_{m_k} - x^\dagger \|^2) + \delta_k^2 \nonumber \\
& & + 2 \delta_k \| F'(x^\dagger) \|  \| x_{m_k} - x^\dagger \| + 2 \delta_k o ( \| x_{m_k}- x^\dagger \|) \\
& & + \alpha_k \| x_{m_k} - x^* \|^2. \nonumber
\end{eqnarray}
By dividing by $\alpha_k$ both sides of \eqref{eq: limite norma x_m, alpha^ dagger} and by using assumptions $i)$ and $iii)$, we have that 
\begin{equation}\label{eq:x_m_alpha_k meno x_star converges}
\| x_{m_k, \alpha_k}^{\delta_k} - x^* \| \to \| x^\dagger - x^* \| ~.
\end{equation}
Therefore, by using the reverse triangle inequality $| \|x_{m_k, \alpha_k}^{\delta_k} \| - \| x^* \| | \le \| x_{m_k, \alpha_k}^{\delta_k} - x^* \|$, we conclude that the sequence $\{x_{m_k, \alpha_k}^{\delta_k}\}_{k \in \mathbb N}$ is bounded.
Then, up to take a subsequence, there exists an element $w \in \mathcal X$ such that
\begin{equation}
    \label{eq: debole convegrenze x mk, ak  deltak}
    x_{m_k, \alpha_k}^{\delta_k} \rightharpoonup w ~~ \text{as} ~~ k \to \infty ~.
\end{equation}
By comparing the right end sides of \eqref{eq: x_mk, alohak, deltak limitata} when $k \to \infty$, we obtain
\begin{equation}
\label{eq: convergenze ai dati}
F(x_{m_k, \alpha_k}^{\delta_k}) \to y ~~ \text{as} ~~ k \to \infty ~.
\end{equation}
Since $C$ and $F$ are weakly closed, we obtain that $w \in D(F) \cap C$ and $F(w)=y$.

It remains to prove that $w$ minimizes the distance from the a priori solution $x^*$.
By using \eqref{eq:x_m_alpha_k meno x_star converges} and the lower semicontinuity of the norm we obtain 
\begin{equation}
\| w - x^* \| \le \limsup_{k \to \infty} \| x_{m_k, \alpha_k}^{\delta_k} - x^* \| = \| x^\dagger - x^* \| \le \|w - x^* \| ~.
\end{equation}
This implies that $w$ is a $x^*$-minimum norm solution. 
Moreover, since
\begin{equation*}
\| x_{m_k, \alpha_k}^{\delta_k} - w \|^2 = \| x_{m_k, \alpha_k}^{\delta_k} - x^* \|^2 + \|x^* - w \|^2 + 2 \langle x_{m_k, \alpha_k}^{\delta_k} - x^*, x^* - w \rangle ~,
\end{equation*}
we have
\begin{equation*}
\limsup_{k \to \infty} \| x_{m_k, \alpha_k}^{\delta_k} - w \|^2 \le 2 \| x^* - w \|^2 - 2 \lim_{k \to \infty} \langle x_{m_k, \alpha_k}^{\delta_k} - x^*, w- x^* \rangle = 0 ~,
\end{equation*}
and then $\{x_{m_k, \alpha_k}^{\delta_k}\}_{k \in \mathbb N}$ strongly converges to $w$.
\end{proof}

We note that, in the case the $x^*$-minimum norm solution is unique, $x_{m(\delta), \alpha(\delta)}^{\delta}$ strongly converges to $x^\dagger$.
Furthermore, hypothesis $iii)$ is a mild requirement for the rate of convergence, since the existence of the sequence is guaranteed by property \eqref{prop: proprietà}.
We now prove the convergence rates of the regularization method.
\begin{thm}
\label{thm: rate di convergenza}
Consider the assumptions $A1$, $A2$ and $A3$. Let $y^\delta \in \mathcal Y$ be such that $\|y^\delta - y\| \le \delta$. 
Assume the following condition are satisfied:
\begin{enumerate}
\item F is Fréchet-differentiable;
\item there exist $\gamma \ge 0$ such that $\| F'(z) - F'(x^\dagger)  \| \le \gamma  \|z- x^\dagger \|$ for all $z \in D(F)$ in a sufficiently large ball $U$ around $x^\dagger$;
\item there exists $w \in \mathcal Y$ such that $x^\dagger - x^* = F'(x^\dagger)^* w$ such that $\gamma \|w \| <  1$.
\end{enumerate}
 
Let consider a function $m : \mathbb{R} \to \mathbb{N}$ such that 
\begin{eqnarray}
\label{hp: def velocita' di convergenza compatti}
m(\delta) \to \infty & \quad \text{and} \quad \|x_{m(\delta)} - x^\dagger\| \to 0 \quad \text{as} \quad \delta \to 0 ~,
\end{eqnarray}
where $x_{m(\delta)} \in \mathcal{X}_{m(\delta)}$. \\
Then, by denoting $\|x_{m} - x^\dagger\|^2 = \mathcal{O}(f(m))$, for $\alpha \propto \delta$ and if $\frac{f(m(\delta))}{\delta} \to 0$ for $\delta \to 0$, we obtain that $x^\dagger$ that satisfies assumptions $2$ and $3$ is unique and
\begin{equation}
\| F(x_{m(\delta), \alpha(\delta)}^\delta) - y^\delta \| \le \max \left\{ \mathcal O (\delta), \mathcal O\left(\delta^{1/2} \left[f(m(\delta))\right]^{1/4} \right)\right\}
\end{equation}
and
\begin{equation}
\| x_{m(\delta), \alpha(\delta)}^\delta - x^\dagger \| \le \max  \left\{ \mathcal O \left(\delta^{1/2}\right), \mathcal{O}\left(\delta^{-1/2}\left[f(m(\delta))\right]^{1/2}\right)\right\} ~.
\end{equation}
\end{thm}
\begin{proof}
Denote with $\xi_m \coloneqq x_{m(\delta)} - x^\dagger$, $\xi_m^* \coloneqq x_{m(\delta)} - x^*$, and $J \coloneqq \|F'(x^\dagger)\|$.
Denote the family of Tikhonov solutions
$x_{m, \alpha}^\delta$, where the dependency of $m$ and $\alpha$ on $\delta$ is omitted.
Since $x_{m} \to x^\dagger$ as $m \to \infty$ there exists an $M$ such that $x_m$ belongs to $U$ for every $m > M$. 
Then, by using the Fréchet differentiability of $F$ we have
\begin{equation}
\label{eq: dis frechet}
    F(x_m)=F(x^\dagger) + F'(x^\dagger)\,\xi_m + r_m~, \quad\text{with}~ \|r_m \| \le \frac{\gamma}{2} \| \xi_m \|^2 ~.
\end{equation}
By using the triangle inequality, we obtain
\begin{eqnarray}
\label{eq: dis minimzzazione}
\| F(x_{m, \alpha}^{\delta}) - y^\delta \|^2 + \alpha \| x_{m, \alpha}^\delta - x^* \|^2 &\le & \| F(x_m) - y^\delta \|^2 + \alpha \| x_m- x^* \|^2  \nonumber \\
&\le & (\| F(x_m) - y \| + \delta)^2 + \alpha \|\xi_m^* \|^2 ~.
\end{eqnarray}
We show that the family $x_{m, \alpha}^\delta$ eventually belongs to the open neighborhood defined in the assumption 2.
By considering
\begin{equation}
\| x_{m, \alpha}^\delta - x^* \|^2 = \| x_{m, \alpha}^\delta - x_m \|^2 + \| x_m - x^* \|^2 + 2 \langle x_{m, \alpha}^\delta - x_m, x_m - x^* \rangle  ~,
\end{equation}
and the Cauchy-Schwarz inequality, we obtain
\begin{eqnarray}
\|F(x_{m, \alpha}^\delta) - y^\delta \|^2 + \alpha \|x_{m, \alpha}^\delta - x_m \|^ 2 
\!\! & \le & \!\!
(\|F(x_m) - y \| + \delta)^2 - 2 \alpha \langle x_{m, \alpha}^\delta - x_m, \xi_m ^* \rangle  \nonumber \\
\!\! & \le & \!\!
(\|F(x_m) - y \| + \delta)^2 + 2 \alpha \| \xi_m^* \|\|x_{m, \alpha}^\delta - x_m \| ~.
\end{eqnarray}
An upper bound of $\|F(x_m) - y\|$ is obtained by using \eqref{eq: dis frechet} and the triangle inequality, as follows
\begin{eqnarray*}
\alpha \|x_{m, \alpha}^\delta - x_m \|^ 2 &\le &  
J^2 \|\xi_m \|^2 + \frac{\gamma^2}{4} \| \xi_m\| ^ 4 + \gamma J \|\xi_m\|^3 + \delta^2 + 2 \delta J \|\xi_m \| \\
& & + \delta \gamma \| \xi_m\|^2 + 2 \alpha \|\xi_m^* \|\|x_{m, \alpha}^\delta - x_m \| ~.
\end{eqnarray*}
Therefore, by completing the square in the variables $\|x_{m, \alpha}^\delta - x_m \|$ and $\|\xi_m^* \|$ we have
\begin{eqnarray}
\|x_{m, \alpha}^\delta - x_m \| & \le & \|\xi_m^* \| + 
\bigg( \|\xi_m^* \|^2 + \frac{ J^2 \|\xi_m \|^2 }{\alpha} + \frac{ \gamma^2  \| \xi_m \|^4}{4 \alpha}
\nonumber \\
& & + \frac{ \gamma J \|\xi_m \|^3}{\alpha} + \frac{\delta^2}{\alpha} + \frac{2 \delta J \|\xi_m \|}{\alpha} + \frac{\delta \gamma \|\xi_m \|^2}{\alpha} \bigg)^{\frac{1}{2}} ~.
\end{eqnarray}
By using assumption 5 and thanks to the choice $\alpha \propto \delta$, we have that $x_{m, \alpha}^\delta \in B_\rho(x^\dagger)$ for any fixed $\rho > 2 \| x^\dagger - x^* \| $.
Without loss of generality, we can assume that $U \supset B_\rho(x^\dagger)$. 

By plugging
\begin{equation}
\| x_{m, \alpha}^\delta - x^* \|^2 = \| x_{m, \alpha}^\delta - x^\dagger \|^2 + \|x^\dagger - x^* \|^2 + 2 \langle x_{m, \alpha}^\delta - x^\dagger, x^\dagger - x^* \rangle, 
 \end{equation}
into \eqref{eq: dis minimzzazione} and by applying the triangle inequality $ \|x_m - x^* \| \le \|x_m - x^\dagger \| + \|x^\dagger - x^* \|$, we obtain
\begin{eqnarray}
\label{eq: studio Tikhonov}
\| F(x_{m, \alpha}^\delta) - y^\delta \|^2 + 
\alpha \| x_{m, \alpha}^\delta - x^\dagger \|^2  
& \le &
(\|F(x_m ) - y\| + \delta)^2 \nonumber \\
& & 
+ \alpha \left( \|\xi_m \|^2 + 2 \|\xi_m \| \|x^\dagger - x^* \| \right)\\
& & + 2 \alpha |\langle x_{m, \alpha}^\delta - x^\dagger, x^\dagger - x^* \rangle | ~. \nonumber
\end{eqnarray}
By using assumption $3$ and by replacing $x_m$ with $x_{m, \alpha}^\delta$ and $r_m$ with $r_{m, \alpha}^\delta$ in \eqref{eq: dis frechet}, we obtain
\begin{eqnarray}
\label{eq:scalar-product}
2 \alpha | \langle x_{m, \alpha}^\delta - x^\dagger, x^\dagger - x^* \rangle | & = & 2 \alpha | \langle F'(x^\dagger)(x_{m, \alpha}^\delta - x^\dagger), w \rangle | \nonumber \\
& = & 2 \alpha | \langle F(x_{m, \alpha}^\delta) - F(x^\dagger) - r_{m, \alpha}^\delta, w \rangle | \nonumber \\
& \le &  2 \alpha \|w \| \|F(x_{m, \alpha}^\delta) - y^\delta \|  + 2 \alpha \delta \|w \| \\
& & + \alpha \gamma \| w \| \| x_{m, \alpha}^\delta - x^\dagger \|^2 \nonumber ~.
\end{eqnarray}
By plugging \eqref{eq:scalar-product} in \eqref{eq: studio Tikhonov}, by using \eqref{eq: dis frechet} and by adding $\alpha^2 \|w \|^2$ to both members, we obtain
\begin{gather*}
    (\|F(x_{m, \alpha}^\delta) - y^\delta \| - \alpha \|w \|)^2 + \alpha(1-\gamma \|w \| )\|x_{m, \alpha}^\delta - x^\dagger \|^2 \le \\
    (\delta + \alpha \|w \|)^2 + (2 \delta J + 2 \alpha \|x^\dagger - x^* \| ) \|\xi_m \| + (\gamma \delta + \alpha + J^2) \|\xi_m \|^2 + o(\|\xi_m \|^2)  ~.
\end{gather*}
Then we conclude the rate of convergence for the data  and the solution: 
\begin{eqnarray}
    \| F(x_{m, \alpha}^\delta)-y^\delta\| & \le  & \alpha \|w\| + \nonumber \\ 
    & & \bigg(
    (\delta + \alpha \|w \|)^2 + (2 \delta J + 2 \alpha \|x^\dagger - x^* \| ) \|\xi_m \| \\
    & & + (\gamma \delta + \alpha + J^2) \|\xi_m \|^2 + o(\|\xi_m \|^2) \bigg)^{1/2} ~, \nonumber
\end{eqnarray}
and
\begin{eqnarray}
\|x_{m, \alpha}^\delta - x^\dagger \| &\le&  
 \frac{1}{(\alpha(1 - \gamma \|w \|))^{1/2}} \bigg( (\delta + \alpha \|w \|)^2 + (2 \delta J + 2 \alpha \|x^\dagger - x^* \| ) \|\xi_m \| \nonumber\\
 & & + (\gamma \delta + \alpha + J^2) \|\xi_m \|^2 + o(\|\xi_m \|^2) \bigg)^{1/2}  ~.
\end{eqnarray}
The thesis follows from these two inequalities.
\end{proof}

The derived convergence rates are identical to those achieved by Ritz’s regularization method for linear operators on nested subspace sequences $\cite{groetsch1984theory}$.
The theorem also shows that, if the subsequence of compact sets converges in such a way that $f(m(\delta))$ is at least $o(\delta^2)$, this method has the same convergence rates of Tikhonov's classical regularization. 
However, if $f(m(\delta))$ converges faster than $\delta^2$, the error rate does not improve.
Finally, we observe that the regularized solution $x_{m(\delta), \alpha(\delta)}^\delta$ belongs at each iteration to the weakly closed set $C$. \\

We now prove that the proposed method achieves the convergence rate $\mathcal{O}(\delta^{2/3})$ for the solution, under the classical  regularity assumption $x^\dagger - x^* \in R((F'(x^\dagger)^*F'(x^\dagger)^\mu$), where $\mu \in [1/2, 1]$.
\begin{thm}
\label{thm: convergenza rate ottimale}
Assume all the hypotheses of theorem \ref{thm: rate di convergenza} hold. Let be $x^\dagger$ is an element in the interior of $D(F) \cap C$ and let $x^\dagger$ satisfies the above mentioned regularity property.
Let us consider a function $m(\delta): \mathbb R \to \mathbb N$ such that 
\begin{eqnarray}
    m(\delta) \to \infty \quad \text{and} \quad \|z_{m(\delta)} - z_\alpha \| \to 0 \quad \text{as} \quad \delta \to 0
\end{eqnarray} 
with $z_{m(\delta)} \in \mathcal{X}_{m(\delta)}$ and  
\begin{equation}
  z_\alpha \coloneqq x^\dagger - \alpha\left( F'(x^\dagger)^* F'(x^\dagger) + \alpha I \right)^{-1}F'(x^\dagger)^* w ~
\end{equation}
where $I$ represent the identity on $\mathcal X$. Let us denote $f(m(\delta)) \coloneqq \|z_{m(\delta)}-z_\alpha\|^2$. Then, if the parameter choice rules satisfy
\begin{eqnarray}
\label{eq: convergenza a x}
\alpha(\delta) = \mathcal O\left(\delta^{\frac{2}{2\mu + 1}}\right) & \quad \text{and} \quad \dfrac{f(m(\delta))}{\delta^{\frac{2}{2\mu + 1}}} \to 0 \quad \text{as} \quad \delta \to 0 ~,
\end{eqnarray}
we obtain that $x^\dagger$ is unique and
\begin{equation}
\| x_{m(\delta), \alpha(\delta)}^\delta - x^\dagger \| \le \max  \left\{ \mathcal O\left(\delta^{\frac{2\mu}{2\mu + 1}}\right), \mathcal{O}\left(\delta^{-\frac{1}{2(2\mu + 1)}}f(m(\delta))^{1/4}\right)\right\} ~.
\end{equation}
\end{thm}
\begin{proof}
The idea is to compare the regularized solution $x_{m,\alpha}^\delta$ with the auxiliary element $z_\alpha$. 
We begin by considering \eqref{eq: dis minimzzazione} where we substitute $x_m$ with $z_m$ since $z_m \in \mathcal X_m$. By adding and subtracting $x^\dagger$ to $x^*$ in the regularization term, we get
\begin{eqnarray}
\label{eq: distanza da x dagger}
    \|x_{m, \alpha}^\delta - x^\dagger \| ^2 & \le&  \frac{1}{\alpha}\left(\|F(z_{m}) - y^\delta \|^2 - \|F(x_{m, \alpha}^\delta) - y^\delta \|^2 \right) \nonumber \\ 
    & & + \| z_m - x^\dagger \|^2 + 2 \langle z_m - x^\dagger, x^\dagger - x^* \rangle  \nonumber \\
     && - 2 \langle x_{m, \alpha}^\delta - x^\dagger, x^\dagger - x^* \rangle ~.
\end{eqnarray}
Following the same steps as in the proof of theorem \ref{thm: rate di convergenza}, we obtain that \eqref{eq: dis frechet} also holds when $x_m$ and $r_m$ are replaced by $x_{m,\alpha}^\delta$ and $r_{m,\alpha}^\delta$, respectively.
Moreover, since $\|z_\alpha - x^\dagger\| \le \sqrt{\alpha} \|w\|$ and $x^\dagger$ lies in the interior of $D(F) \cap C$, it follows that $z_\alpha \in D(F) \cap C$ for sufficiently small $\alpha$, and hence \eqref{eq: dis frechet} also holds 
when $x_m$ and $r_m$ are replaced by $z_\alpha$ and $s_\alpha$, respectively. As a consequence, by using the Fréchet differentiability of $F$ at $z_\alpha$ and some simple algebra, we can write 
\begin{eqnarray}
    \| F(x_{m, \alpha}^\delta) - y^\delta \|^2 & = &  \| F(x_{m, \alpha}^\delta) - y^\delta +\alpha w \|^2  - 2 \langle F'(x^\dagger) (x_{m, \alpha}^\delta - x^\dagger) + r_{m, \alpha}^\delta , \alpha w \rangle \nonumber\\
    && -  \langle 2( y - y^\delta) + \alpha w, \alpha w\rangle \nonumber 
\end{eqnarray}
and
\begin{eqnarray*}
    \|F(z_m) - y^\delta\|^2 & = & \|F'(z_\alpha)(z_m - z_\alpha) + s_{m,\alpha}\|^2 + \|F'(x^\dagger)(z_\alpha - x^\dagger) \|^2 \\
    && + \| s_\alpha +y - y^\delta \|^2  + 2 \langle F'(x^\dagger) (z_\alpha - x^\dagger), s_\alpha + y - y^\delta \rangle   \\
    && + 2 \langle F(z_m) - F(z_\alpha), F(z_\alpha) - y \rangle + 2 \langle F(z_m) - F(z_\alpha), y - y^\delta \rangle~,
\end{eqnarray*}
where $s_{m,\alpha}$ is such that 
\begin{eqnarray*}
    F(z_m) = F(z_\alpha) + F'(z_\alpha) (z_m - z_\alpha ) + s_{m,\alpha} \quad \text{with} ~ \|s_{m,\alpha}\| = o(\|z_m - z_\alpha\|) ~.
\end{eqnarray*}
By plugging these two terms in \eqref{eq: distanza da x dagger} and by using the assumption $3$, we get
\begin{eqnarray}
    \|x_{m, \alpha}^\delta - x^\dagger \|^2 &\le& \frac{1}{\alpha}  \|F'(z_\alpha)(z_m - z_\alpha) + s_{m,\alpha}\|^2 + \frac{1}{\alpha} \|F'(x^\dagger)(z_\alpha - x^\dagger) \|^2 + \frac{1}{\alpha}  \| s_\alpha +y - y^\delta \|^2 \nonumber \\
    && + \frac{2}{\alpha} \langle F'(x^\dagger) (z_\alpha - x^\dagger), s_\alpha + y - y^\delta \rangle  \nonumber \\
    && + \frac{2}{\alpha}  \langle F(z_m) - F(z_\alpha), F(z_\alpha) - y \rangle + \frac{2}{\alpha}  \langle F(z_m) - F(z_\alpha), y - y^\delta \rangle  \nonumber \\
    && +  \langle 2( y - y^\delta) + \alpha w + r_{m, \alpha}^\delta , w\rangle + \| z_m - x^\dagger \|^2 + 2 \langle F'(x^\dagger)(z_m - z_\alpha), w \rangle \\
    && + 2 \langle F'(x^\dagger)(z_\alpha - x^\dagger), w \rangle \nonumber ~. 
\end{eqnarray}
Following \cite{engl1996regularization}, we use the relations
\begin{eqnarray*}
    \frac{1}{\alpha} \| F'(x^\dagger) (z_\alpha - x^\dagger) \|^2 + \langle 2F'(x^\dagger) (z_\alpha - x^\dagger) + \alpha w, w\rangle  &= &\alpha^3 \| (F'(x^\dagger) F'(x^\dagger)^* + \alpha I)^{-1} w\|^2 ~, \\
    \frac{1}{\alpha} \|s_\alpha + y - y^\delta \|^2 &\le& \frac{2}{\alpha} \left( \frac{\gamma}{4}\| z_\alpha - x^\dagger \|^4 + \delta^2 \right)  ~, \\
    2 \langle  y - y^\delta, w \rangle +\frac{2}{\alpha}\langle F'(x^\dagger) (z_\alpha - x^\dagger), s_\alpha + y - y^\delta\rangle &\le& \gamma \|w\| \|z_\alpha - x^\dagger\|^2 \\
    && + 2 \alpha \delta \| (F'(x^\dagger) F'(x^\dagger)^* + \alpha I)^{-1} w\|^2 ~,
\end{eqnarray*}
and we obtain
\begin{eqnarray*}
(1 - \gamma \|w\|)  \|x_{m, \alpha}^\delta - x^\dagger \|^2 &= & \mathcal O \left(\alpha^3 \| (F'(x^\dagger) F'(x^\dagger)^* + \alpha I)^{-1} w\|^2 \right) \\
&& + \mathcal O \left(\frac{\| z_\alpha - x^\dagger \|^4 + \delta^2}{\alpha}  \right) + \mathcal{O}\left( \|z_\alpha - x^\dagger \|^2 \right)\\
&& + \mathcal O \left( \alpha \delta \| (F'(x^\dagger) F'(x^\dagger)^* + \alpha I)^{-1} w\|^2\right) \\
&& +  \mathcal O \left(\frac{\|z_m - z_\alpha\|^2}{\alpha} \right)+ o \left(\frac{\|z_m - z_\alpha\|^2}{\alpha} \right) \\
&& + \mathcal O \left(\frac{\|z_m - z_\alpha \|\|z_\alpha - x^\dagger\|}{\alpha}\right) + o\left(\frac{\|z_m - z_\alpha \|\|z_\alpha - x^\dagger\|}{\alpha}\right) \\
&& + \mathcal O \left(\frac{\|z_m - z_\alpha \|\delta}{\alpha}\right) + o\left(\frac{\|z_m - z_\alpha \|\delta}{\alpha}\right) \\
&& + \mathcal O\left( \|z_m - z_\alpha \| \right) + \mathcal O \left( \|z_m - z_\alpha \|^2 \right) + \mathcal O \left( \|z_\alpha - x^\dagger\| \|z_m - z_\alpha \|\right) ~.
\end{eqnarray*}
The first four terms are controlled by the classical convergence rate, while the others are governed by term $\mathcal O\left(\frac{\|z_m - z_\alpha \|}{\sqrt{\alpha}}\right)$. This implies the thesis.
\end{proof}
As in theorem  \ref{thm: rate di convergenza}, this method has the same convergence rates of Tikhonov regularization if the rate $f(m(\delta))$ of the compact sets  is at least $\mathcal O(\delta^{10/3})$. However, no benefit is obtained even when the rate is faster.

\subsection{Linear case}
In this section, we show how the theoretical results presented in section \ref{sec:Tikhonov-regularization} can be derived for a linear forward operator under milder assumptions. 
Hereafter, the operator $F$ between $\mathcal{X}$ and $\mathcal{Y}$ will be linear and denoted by $A$.
In order to regularize problem \eqref{eq: F(x)=y}, we define a sequence of compact sets $\{\mathcal X_m\}_{m \in N}$ that satisfies \eqref{prop: proprietà}.

\begin{assumption}
We make the following assumptions. 
\begin{enumerate}
\item[A1'.]  $A$ is a bounded operator.
\item[A2'.] $D(A)$ is weakly closed set.
\item[A3'.] $C$ is weakly closed, and there exists a solution of the problem $Ax=y$ such that $x \in C \cap D(A)$.
\end{enumerate}
\end{assumption}

In the linear case, we take $x^* = 0$. For simplicity, we will refer to $x^\star$-minimum norm solutions (see \eqref{def: x dagger}) as \textit{minimum norm solutions}.
The definition of regularized solution in the linear case is a particular case of \eqref{def: x m alpha delta}.

\begin{rem}
\label{prop: linear operato weakly closed}
Under assumptions $A1'$ and $A2'$, the operator $A$ is weakly closed.
Therefore, the existence of minimum norm solutions and regularized solutions is guaranteed by Proposition \ref{prop: esistenza di x dagger} and Weierstrass Theorem.
In addition, since theorems \ref{thm: stabilità Tikhonov in xm} and \ref{thm: convergenza to x dagger} are true in the linear case, the stability and convergence of the proposed regularization method are still valid.
\end{rem}
We establish below that the convergence rates derived in the non-linear setting (theorems \ref{thm: rate di convergenza} and \ref{thm: convergenza rate ottimale}) remain valid under less restrictive assumptions.
\begin{thm}
\label{thm: conveergenza rate classico lineare}
Let assumptions $A1’$, $A2’$, and $A3’$ hold. Suppose $y^\delta \in \mathcal{Y}$ satisfies $\|y^\delta - y\| \leq \delta$, and assume that $x^\dagger \in \mathcal{R}(A^*)$. 
Let consider a function $m: \mathbb R \to \mathbb N$ such that 
$$m(\delta) \to \infty  \quad \text{and} \quad \dfrac{\|x_{m(\delta)} - x^\dagger\|^2}{\delta} \to 0 \quad \text{as} \quad \delta \to 0 $$
where $x_{m(\delta)} \in \mathcal{X}_{m(\delta)}$. Then, by denoting $\|x_m - x^\dagger\|^2 = \mathcal O(f(m))$, for $\alpha \propto \delta$ we obtain that $x^\dagger \ R(A^*)$ is unique and we obtain 
\begin{equation}
\| Ax_{m(\delta), \alpha(\delta)}^\delta - y^\delta \| \le \max \left\{ \mathcal O (\delta), \mathcal O\left( f(m(\delta))^{1/2} \right)\right\}
\end{equation}
and
\begin{equation}
\| x_{m(\delta), \alpha(\delta)}^\delta - x^\dagger \| \le \max  \left\{ \mathcal O \left(\delta^{1/2}\right), \mathcal{O}\left(\delta^{-1/2}\left[f(m(\delta))\right]^{1/2}\right)\right\} ~.
\end{equation}
\end{thm}

\begin{proof}
We denote $\xi_m \coloneqq x_m - x^\dagger$ and we omit dependence of $m$ and $\alpha$ on $\delta$ in the notation of the regularized solutions, i.e.,  $ \{x^\delta_{m, \alpha}\} \coloneqq \{x^\delta_{m(\delta), \alpha(\delta)}\}$.
Since $x_{m, \alpha}^\delta$ minimize the Tikhonov functional in $\mathcal X_m$ we obtain:
\begin{eqnarray}
\label{eq: minimization linear case}
\| Ax_{m, \alpha}^\delta - y^\delta \|^2 + \alpha \| x_{m, \alpha}^\delta \|^2 & \le &  \| A x_m - y^\delta \|^2 + \alpha \|x_m \|^2  \nonumber \\
& \le &  \|Ax_m - y \|^2 + \delta^2 + 2 \delta \|Ax_m - y \| + \alpha \|x_m \|^2 \\
& \le &  \|A \|^2 \| \xi_m \|^2 + \delta^2 + 2 \delta \|A \| \|\xi_m \| + \alpha \|x_m \|^2  \nonumber ~.
\end{eqnarray}
By using the two relationships: 
\begin{gather*}
    \| x_{m, \alpha}^\delta \|^2 = \|x_{m, \alpha}^\delta - x^\dagger \|^2 + \|x^\dagger\|^2 + 2 \langle x_{m, \alpha}^\delta - x^\dagger, x^\dagger \rangle ~, \\
    \| x_m \|^2 = \| x_m - x^\dagger \|^2 + \|x^\dagger \|^2 + 2 \langle x_m - x^\dagger, x^\dagger \rangle ~,
\end{gather*}
and the Cauchy-Schwarz inequality we have:
\begin{eqnarray}
\label{eq: conti caso lineare}
&&  \| Ax_{m, \alpha}^\delta - y^\delta \|^2 + \alpha \| x_{m, \alpha}^\delta - x^\dagger\|^2 + 2 \alpha \langle x_{m, \alpha}^\delta - x^\dagger, x^\dagger \rangle \nonumber   \\
&\le& 
 \|A \|^2 \| \xi_m \|^2 + \delta^2 + 2 \delta \|A \| \|\xi_m \|  +  \alpha \|\xi_m \|^2 + 2 \alpha \langle \xi_m, x^\dagger \rangle  \\
& \le & 
 \|A \|^2 \| \xi_m \|^2 + \delta^2 + 2 \delta \|A \| \|\xi_m \|  + \alpha \|\xi_m \|^2 + 2 \alpha \| \xi_m\| \| x^\dagger \| \nonumber ~.
\end{eqnarray}
By using assumption $1$, i.e., there exists $w \in \mathcal Y$ such that $A^* w = x^\dagger$, we obtain: 
 \begin{eqnarray}
 \label{qe: prod scalare caso linare}
\langle x_{m, \alpha}^\delta- x^\dagger, x^\dagger \rangle &  = &  \langle A x_{m, \alpha}^\delta - y , w \rangle \nonumber \\
&  = & \langle Ax_{m, \alpha}^\delta - y^\delta + y^\delta - y, w \rangle \\
&\le & \|A x_{m, \alpha}^\delta
 - y^\delta \| \| w \| + \delta \|w \| 
 \nonumber~.
\end{eqnarray}
By applying \eqref{qe: prod scalare caso linare} in  \eqref{eq: conti caso lineare}, we get 
 \begin{gather*}
 \| Ax_{m, \alpha}^\delta - y^\delta \|^2 - 2 \alpha \|w \| \|A x_{m, \alpha}^\delta - y^\delta \| + \alpha \|x_{m, \alpha}^\delta - x^\dagger \|^2 \le \\
 \|A\|^2 \|\xi_m \|^2 + \delta^2 + 2 \delta \|A \| \|\xi_m \| + \alpha \|\xi_m \|^2 + 2 \alpha \|\xi_m \| \|x^\dagger \| + 2 \alpha \delta \|w \| ~.
 \end{gather*}
Then, by adding the quantity $\alpha^2 \|w \|^2$ to both sides, we obtain 
\begin{gather*}
\left(\|Ax_{m, \alpha}^\delta - y^\delta \| - \alpha \|w \| \right)^2 + \alpha \|x_{m, \alpha}^\delta - x^\dagger \|^2 \le \\
\left( \delta + \alpha \|w \| \right)^2 + \| A \|^2 \|\xi_m \|^2 + 2 \delta \|A \| \|\xi_m \| + \alpha \|\xi_m \|^2 + 2 \alpha \|\xi_m \| \|x^\dagger \| ~.
\end{gather*}
Finally, the following upper bounds
\begin{gather}
\label{eq:convergence}
\| Ax_{m, \alpha}^\delta - y^\delta \| \le \\
\alpha \| w \| + \sqrt{(\delta + \alpha \|w \|)^2 + \|A \| \|\xi_m \|^2 + 2 \delta \|A\|\|\xi_m \|+ \alpha \|\xi_m\|^2 + 2 \alpha \|\xi_m \| \|x^\dagger \|} \nonumber 
\end{gather}
and
\begin{gather}
\label{eq:convergence2}
\| x_{m, \alpha}^\delta - x^\dagger \| \le \\
\sqrt{\frac{(\delta + \alpha \|w \|)^2 + \|A \| \|\xi_m \|^2 + 2 \delta \|A\|\|\xi_m \|+ \alpha \|\xi_m\|^2 + 2 \alpha \|\xi_m \| \|x^\dagger \|}{\alpha}} \nonumber
\end{gather}
yield the thesis. 
\end{proof}
\begin{thm}
\label{thm: rate convergenza ottimale lineare}
    Under assumptions $A1', A2'$ and $A3'$,  let $x^\dagger = (A^A)^\mu v$, where $v \in \mathcal{X}$ and $\mu \in [1/2, 1]$.
    Let us consider a function $m: \mathbb R \to \mathbb N$ such that
    $$
    m(\delta) \to \infty \quad \text{and} \quad \|z_{m(\delta)} - z_\alpha \| \to 0 \quad {as} \quad \delta \to 0
    $$
    with $z_{m(\delta)} \in \mathcal X_{m(\delta)}$ and 
    $$
    z_\alpha \coloneqq x^\dagger - \alpha \left(A^*A + \alpha I\right)^{-1}A^*w ~.$$ 
    Let us denote $f(m(\delta))= \|z_{m(\delta)}-z_\alpha\|^2$. Then, if the parameter choice rules satisfy
    $$
    \alpha(\delta) = \mathcal O\left(\delta^{\frac{2}{2\mu + 1}}\right) \quad \text{and} \quad \frac{f(m(\delta))}{\delta^{\frac{2}{2\mu +1}} } \to 0 \quad \text{as} \quad \delta \to 0 
    $$
    we obtain that $x^\dagger$ is unique and
    \begin{equation}
    \label{ths: convergenza ottimale lin}
    \| x_{m(\delta), \alpha(\delta)}^\delta - x^\dagger \| \le \max  \left\{ \mathcal O\left(\delta^{\frac{2\mu}{2\mu + 1}}\right), \mathcal{O}\left(\delta^{-\frac{1}{2(2\mu + 1)}}f(m(\delta))^{1/4}\right)\right\} ~.
    \end{equation}
\end{thm}
\begin{proof}
The proof follows the same strategy as the proof of theorem \ref{thm: convergenza rate ottimale} by considering $F$ to be linear, as previously done for theorem \ref{thm: conveergenza rate classico lineare}.
\end{proof}

In this case, the bound expressed in equation \eqref{ths: convergenza ottimale lin} states that the proposed method is order optimal.

\section{Application to neural networks}
\label{sec:neural_networks}
In this section we apply the proposed method in the case where $\mathcal X$ is the function space $L^2(\Omega)$, with $\Omega$ a compact subset of $\mathbb R^d$, $C$ is the subset of positive functions\footnote{The subset $C$ is closed and convex, therefore it is weakly closed.},
and the compact sets $\{\mathcal X_m\}_{m\in\mathbb{N}}$ consists of NNs with limited weights.
Based on the approach in \cite{voigtlaender2019approximation} and \cite{petersen2021topological} we introduce the following definitions.  

\begin{defn}
A neural network with  $l$ layers, $d_1,\ldots,d_l$ neurons in the layers $1,\ldots,l$ respectively and activation function $\varphi$ is a map
\begin{equation}
\mathcal{N}(\cdot \ ;\omega) = \left[\mathop{\bigcircle}_{i=1}^l \varphi (W_i \cdot + w^0_i) \right](\cdot) ~,
\end{equation}
where $\omega = (\omega_i)_{i=1}^l$, $\omega_i =(W_i, w_i^0) \in \mathbb{R}^{d_i \times (d_{i-1}+1)}$ is the set of parameters of the $i$-th layer (with $d_0 = d$), $D=\sum_{i=1}^l d_i \times (d_{i-1}+1)$ is the total number of parameters, 
and $\varphi \colon \mathbb{R} \to \mathbb{R}$ is applied component-wise.
The set of NNs with $l$ layers and $d_i$ neurons in the $i$-th layer is denoted by
\begin{equation}
\mathcal{N}_{(d_1, \ldots, d_l)} \coloneqq \left\{ \mathcal{N}(~\cdot~;\omega) ~|~ \omega_i  \in \mathbb{R}^{d_i \times (d_{i-1}+1)} ,~  i=1,\dots,l \right\} ~.
\end{equation}
\end{defn}

We now identify which conditions must be satisfied by the elements of the sets $\mathcal N_{(d_1, \ldots, d_l)}$ to define a sequence of sets $\{\mathcal X_m\}_{m\in\mathbb{N}}$ that is complaint with \eqref{prop: proprietà}.
Building on the universal approximation theorem \cite{kidger2020universal,leshno1993multilayer}, it is well-established that increasing the number of parameters in an NN architecture makes the set of generated functions dense in $L^2(\Omega)$. Moreover, restricting the network weights to a compact set ensures the compactness of the generated functions set. 
Finally, applying the Rectified Linear Unit (ReLU; \cite{Goodfellow2016}),
denoted with $(\cdot)_+~$, to the NN output  enforces the generated functions to be positive.
\begin{thm}
\label{thm: density of nn in positive function}
Let 
$\left\{\left(d_1^m, \ldots, d^m_{l(m)}\right)\right\}_{m \in \mathbb N}$ be a sequence of tuples where $d_i^m \in \mathbb N$ for all $i$ and $m$, and let
$\{c_m\}_{m \in \mathbb N} \subset \mathbb R$ be an increasing sequence.
Let us consider the set of NNs generated by the structure $(d_1^m, \ldots, d_{l(m)}^m)$, with a continuous and non polynomial activation function $\varphi$ and with bounded weights
\begin{align}
    \mathcal N^+_m \coloneqq \Big\{  \left(\mathcal N( \cdot, \omega)\right)_+ ~ \Big| ~ \mathcal{N}(\cdot, \omega) \in \mathcal N_{\left(d_1^m, \ldots, d_{l(m)}^m\right)}, 
    \| \omega_i^m\|_{\infty} \le c_m, ~ i=1, \ldots, l(m) \Big\} ~.
\end{align}
Then $\mathcal N^+_m$ is compact in $L^2(\Omega)$ for every $m \in \mathbb N$. In addition, if we assume that $c_m \to \infty$ and the total number of NN parameters $D^m$ goes to infinity as $m \to \infty$, then
\begin{eqnarray}
    \overline{\bigcup_{m \in \mathbb N} \mathcal{N}^+_m}= L^2(\Omega)^+ ~, 
\end{eqnarray}
where $L^2(\Omega)^+$ denote the set of positive functions in $L^2(\Omega)$.
\end{thm}

\begin{proof}
The first step of the proof is to show that the sets $\mathcal N^+_m$ are compact in $L^2(\Omega)$. From Proposition 3.5 in \cite{petersen2021topological}, we know that, for every $m\in\mathbb{N}$, the set 
$$
\left\{ \mathcal N (\cdot, \omega) \in \mathcal{N}_{\left(d_1^m, \ldots, d_{l(m)}^m \right)} ~|~ \| \omega_i^m \| \le c_m, i=, \ldots, l(m) \right\}
$$
is compact in $C(\Omega)$, i.e., the set of continuous function on $\Omega$, as the weights are bounded within a compact set. 
By composing these functions with $(\cdot)_+$, which is a continuous operator, it follows that $\mathcal{N}^+_m$ is compact in $C(\Omega)$.
Since $C(\Omega)$ is continuously embedded into $L^2(\Omega)$, we conclude that $\mathcal N^+_m$ is compact in $L^2(\Omega)$.

Now we prove that the family $\{ \mathcal N^+_m  \}_{m \in \mathbb N}$ is dense in $L^2(\Omega)^+$. 
Since $\mathcal N^+_m \subseteq L^2(\Omega)^+$  for every $m \in \mathbb N$ and the limit of positive functions is positive, we obtain
\begin{equation}
\label{neural-networks}
    \overline{\bigcup_{m \in \mathbb N} \mathcal N^+_m} \subseteq L^2(\Omega)^+ ~.
\end{equation}
Thanks to the universal approximation theorem \cite{kidger2020universal,leshno1993multilayer}, for every $f \in L^2(\Omega)^+$ there exists an NN $g$ such that
\begin{equation}
    \| f - g \|_{L^2(\Omega)} \le \epsilon ~.
\end{equation}
Since $D^m \to \infty$, either $l(m) \to \infty$ or $d_i^m \to \infty$ for a specific $i$.
In the first case, the existence of such a $g$ is guaranteed by \cite{kidger2020universal}. 
Denoting $g_+ \coloneqq (\cdot)_+ \circ g$, since $l(m) \to \infty$, there exist an $\overline{m} \in \mathbb N$ such that $g_+ \in \mathcal N_{\overline{m}}^+$. 
Similar conclusion can be derived in the case where $d_i^m \to \infty$.
In this case, the existence of such a $g$ is guaranteed by \cite{leshno1993multilayer}.
The thesis follows from 
\begin{equation}
    \| f - g_+ \|_{L^2(\Omega)}
    \leq \| f - g \|_{L^2(\Omega)} \le \epsilon ~.
\end{equation}

\end{proof}

This result shows that the regularization method introduced in section \ref{sec:Tikhonov-regularization} can be applied to the compact sets $\mathcal N_m^+$.
Therefore, in the following we will consider ${\mathcal X}_m = {\mathcal N}^+_m$.
We now introduce a theorem, proved in \cite{voigtlaender2019approximation}, which describes the approximation capabilities of a NN with a fixed number of nodes and bounded weights.
The theorem will be subsequently adopted to define a parameter choice rule $m(\delta)$ leading to the optimal convergence rate ${\mathcal O}(\delta^{2/3})$ \cite{engl1996regularization}.
\begin{thm}
\label{thm: rate di convegrenza reti neurali pesi limitati}
Let $d \in \mathbb N$, $\beta >0$, and $\nu$ be a finite Borel measure on $[-1/2, 1/2]^d $. Then for any bounded function $f \in C^\beta( [-1/2, 1/2]^d )$ there exists a network $\mathcal N (\cdot, \omega) \in \mathcal{N}_{\left(d_1, \ldots, d_{l} \right)}$ with $l \le 7 + (1+ \lceil \log_2(\beta)\rceil)(11+ \beta d)$ and weights in the set
$$ 
[\epsilon^{-s}, \epsilon^s] \cap \{ k \, 2^{-s \left\lceil \log_{2}(1/\epsilon) \right \rceil} \}_{k \in \mathbb Z} 
$$
such that
$$
\|f - \mathcal N(\cdot,\omega) \|_{L^2(\nu)} \le P \epsilon
\quad \text{and} \quad Q \lesssim \|\omega\|_\infty \le P \epsilon^{-d/\beta} ~,
$$
for some $s \in \mathbb N$, $0 < \epsilon < 1/2$ and $P>0$, and where $Q$ denotes the total number of neurons.
\end{thm}
\begin{corollary}
Let us assume that the solution $x$ of \eqref{thm: rate di convegrenza reti neurali pesi limitati} is defined on the set $\left [-1/2, 1/2\right]^d$ and that the hypotheses of theorem \ref{thm: convergenza rate ottimale} are satisfied.
Let us consider $s$ as defined in theorem \ref{thm: rate di convegrenza reti neurali pesi limitati}. 
If the parameters choice rules $\alpha(\delta)$ and $m(\delta)$ are such that
    \begin{itemize}
        \item $\alpha(\delta) = \mathcal O(\delta^{2/3})$,
        \item $c_{m(\delta)} = o(\delta^{-2s/3})$,
        $l(m(\delta))= 7  + (1+ \lceil \log_2(\beta)\rceil)(11+ \beta d)$
        \item $\sum_{i=1}^{l(m(\delta))} d_i = C \delta^{-d2/3\beta}$,
    \end{itemize}
    then we obtain the solution convergence rate
    \begin{equation}
        \|x_{m(\delta), \alpha(\delta)}^\delta - x^\dagger \| = \mathcal O(\delta^{2/3}) ~.
    \end{equation}
\end{corollary}

The proof of this theorem follows directly from the proof of Theorem \ref{thm: rate di convergenza}, with an appropriate reordering of $\alpha(\delta)$ and $f(m(\delta))$.
With this corollary, we thus prove that, for suitable choices of the parameter, our method achieves the optimal convergence rates for solving the inverse problem.
Moreover, this theorem allows us to determine the minimal dimensions of the NN required to attain the optimal convergence rate.

\section{Numerical Experiments}
\label{sec:numexp}
This section is dedicated to a comparison between the standard Tikhonov regularization method and our proposed approach in the case of a synthetic image reconstruction problem.
To explore the regularization capabilities and limitations of the proposed method, we consider the image reconstruction problem of Computerized Tomography \cite{natterer2001mathematics} in which the forward operator is modelled by means of the Radon Transform with limited angles \cite{engl1996regularization}.
We perform numerical experiments to assess, in practice, the convergence properties of the proposed regularization method and to show the reliability of the reconstructions obtained by our approach.

\subsection{Problem Setup}
We consider the inverse problem \eqref{eq: F(x)=y}  where $\mathcal X= L^2(\Omega)$ with $\Omega = [-1,1]^2$, $\mathcal Y=L^2(Z)$ with $Z = [0,\pi)\times[-1,1]$, and $F$ is the Radon transform, and the ground truth object $x$ to reconstruct is the Shepp-Logan phantom \cite{shepp_logan}. 
In the numerical experiments, we consider a discretization of the domain $\Omega$ over $N = 128\times 128$ points $\{(\xi^1_i,\xi^2_j)\}_{i,j}$ with $i,j = 1,\ldots,128$.
The ground truth object $x$ becomes an array ${\bf x}$ with components  ${\bf x}_{ij} \coloneqq x(\xi^1_i,\xi^2_j)$.
We compute the (discrete) Radon transform of $\bf x$ for $50$ different angles between 0 and $\pi$ by means of the DeepInverse software library \cite{deepinv2023}.
The resulting data, called sinogram, is a vector $\bf y$, whose components ${\bf y}_{pq}$ represent the value of the Radon transform of $\bf x$ over a discrete set of $M = 182\times 50$ points $\{(s_p,\theta_q)\}_{pq}$ in $Z$.
Figure \ref{fig:shepp_logan} shows the ground truth image ${\bf x}$ and the corresponding sinogram ${\bf y}$.

In the numerical tests, we consider noisy perturbations ${\bf y}^\delta$ of $\bf y$, in which white Gaussian noise is added, i.e.,
\begin{equation}
{\bf y}^\delta \coloneqq {\bf y} + \delta {\bf n} ~,
\end{equation} 
where ${\bf n}$ is a standard normal random vector and $\delta > 0$.
To regularize the finite-dimensional inverse problem under consideration, we implement both the classical Tikhonov approach and our proposed methodology. 
In the former case, a regularized solution is obtained as
\begin{equation}\label{eq:ip_radon_tikhonov}
{\bf x}^\delta_{\alpha} \coloneqq \arg\min_{{\bf x} \in \mathbb{R}^N} \| \mathcal{R} {\bf x} - {\bf y}^\delta \|^2 + \alpha \| {\bf x} \|^2  = (\mathcal{R}^T\mathcal{R} + \alpha I)^{-1} \mathcal{R}^T {\bf y}^\delta ~, 
\end{equation}
where $\mathcal{R}$ represents the discretized Radon transform.
As for our proposed technique, we solve the optimization problem 
\begin{equation}\label{eq:ip_radon_NN}
{\bf x}^\delta_{\omega^\ast, \alpha} \coloneqq \arg\min_{{\bf x}_{\omega} \in C} \| \mathcal{R} {\bf x}_\omega - {\bf y}^\delta \|^2 + \alpha \| {\bf x}_\omega \|^2 ~,
\end{equation}
where $C \coloneqq \mathbb{R}^N_+$ is the non-negative orthant and the components of the vector ${\bf x}_\omega$ are obtained by evaluating a NN function parameterized by an array of weights $\omega$ over the grid of points in $\Omega$, viz.
$$
({\bf x}_{\omega})_{ij} \coloneqq {\mathcal N}(\xi^1_i,\xi^2_j; \omega) ~.
$$
for any $\alpha$. Specifically, problem \eqref{eq:ip_radon_NN} consists in finding the optimal set of weights $\omega^\ast$ of the NN given the data ${\bf y}^\delta$ and the grid of points for any $\alpha$.
In our experiments, $\mathcal N$ is a Multi-layer Perceptron (MLP; \cite{Goodfellow2016}) with $4$ hidden layers and $100$ neurons per layer.
The NN takes as input the two coordinates of the grid of points $(\xi^1_i,\xi^2_j)$ and provides as output an estimate of the intensity of the (unknown) object $x$ at that point.
To satisfy the constraint $C$, a ReLU activation function is applied in the final layer of  $\mathcal{N}$; the activation function adopted in the hidden layers is instead leaky ReLU \cite{maas2013rectifier}. 
The MLP implementation is performed by means of the pytorch library \cite{paszke2017automatic}, and problem \eqref{eq:ip_radon_NN} is solved for any $\alpha$ by means of the Adam optimizer \cite{kingma2017adammethodstochasticoptimization}.

\begin{figure}[t]
    \centering
    \includegraphics[width=0.4\textwidth]{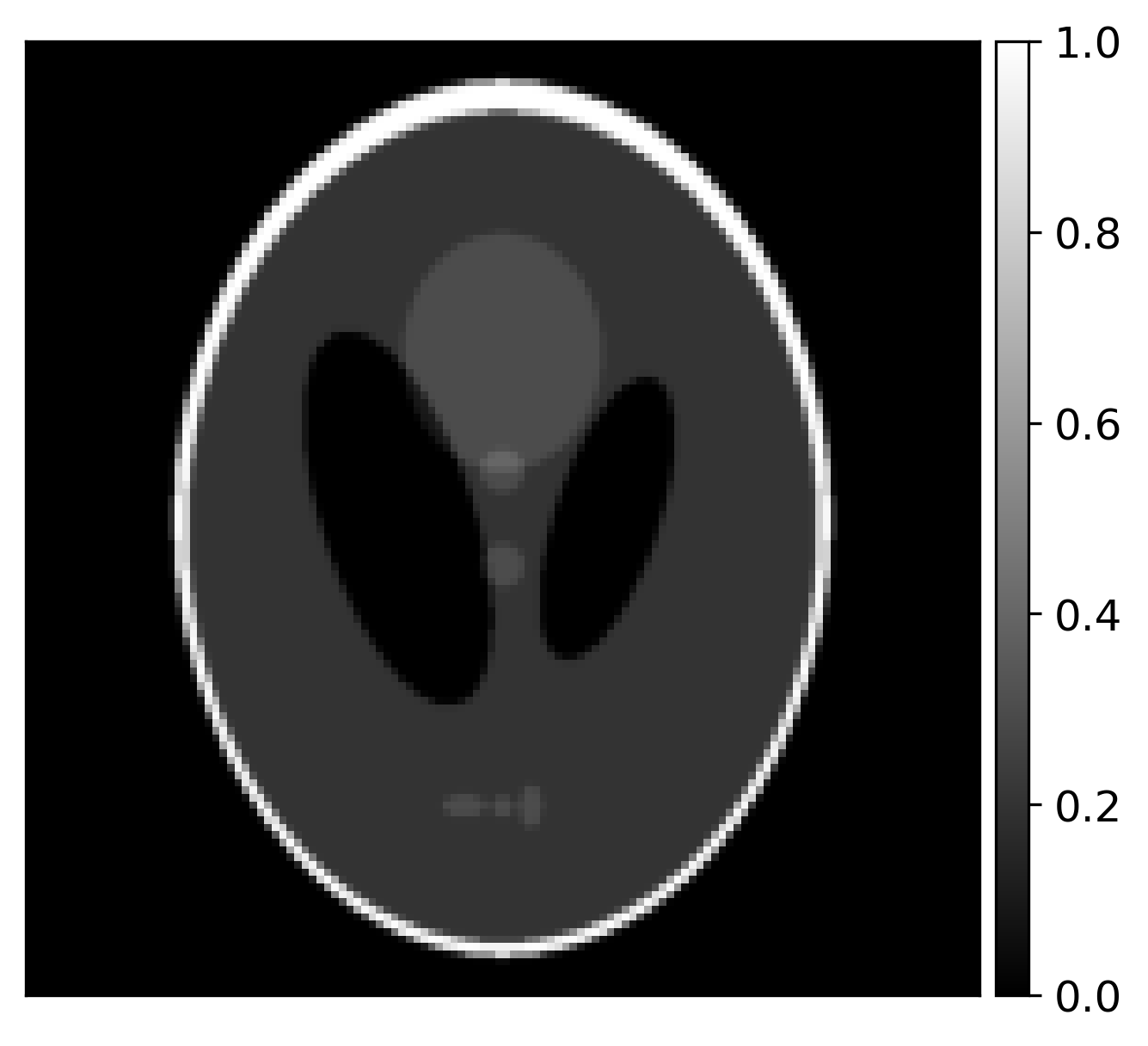}
    \hspace{10pt}
    \includegraphics[width=0.4\textwidth]{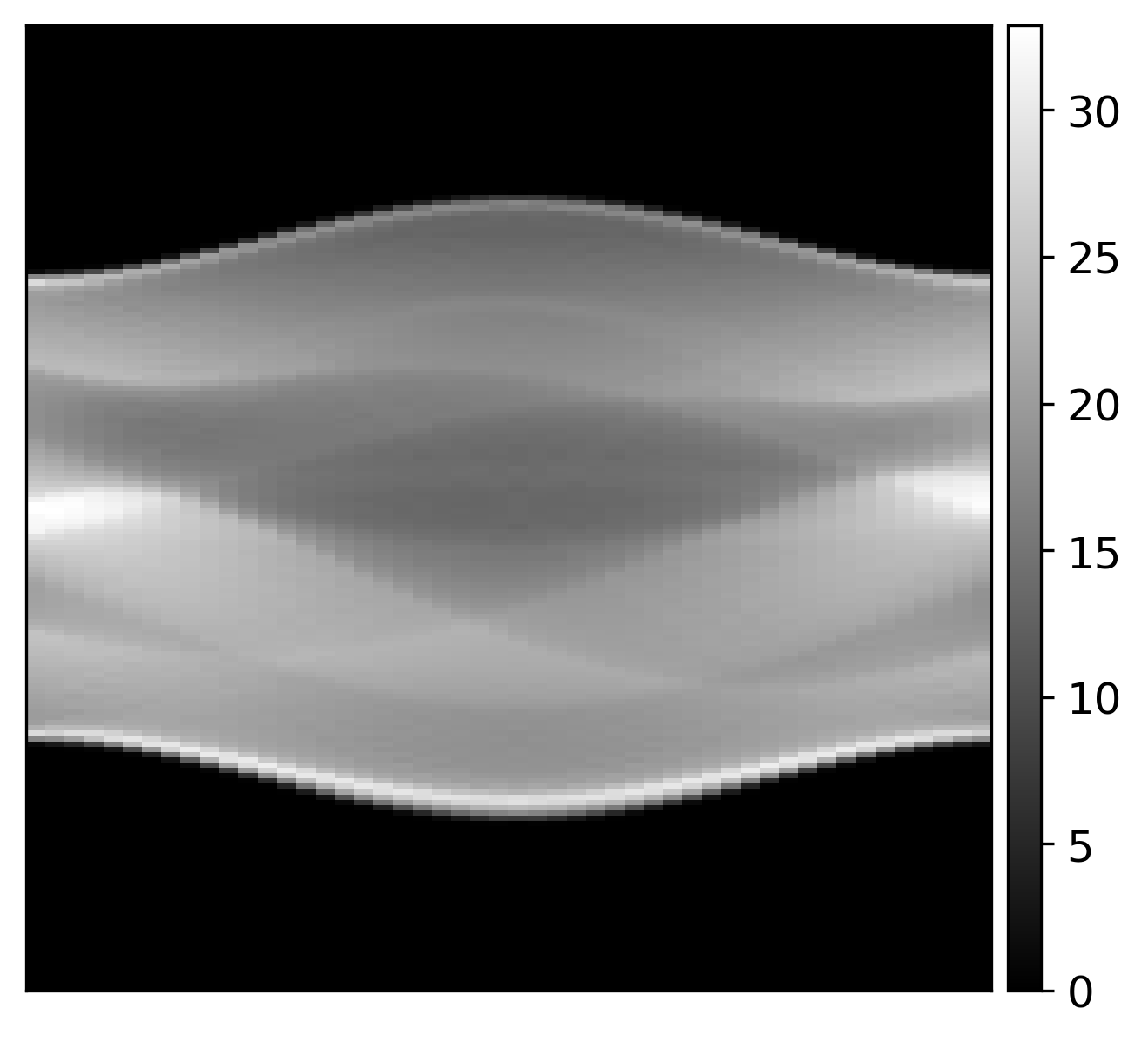}
    \caption{Shepp-Logan phantom and corresponding sinogram (left and right panel, respectively).}
    \label{fig:shepp_logan}
\end{figure}

\begin{figure}[t] 
    \centering
    \includegraphics[width=\textwidth]{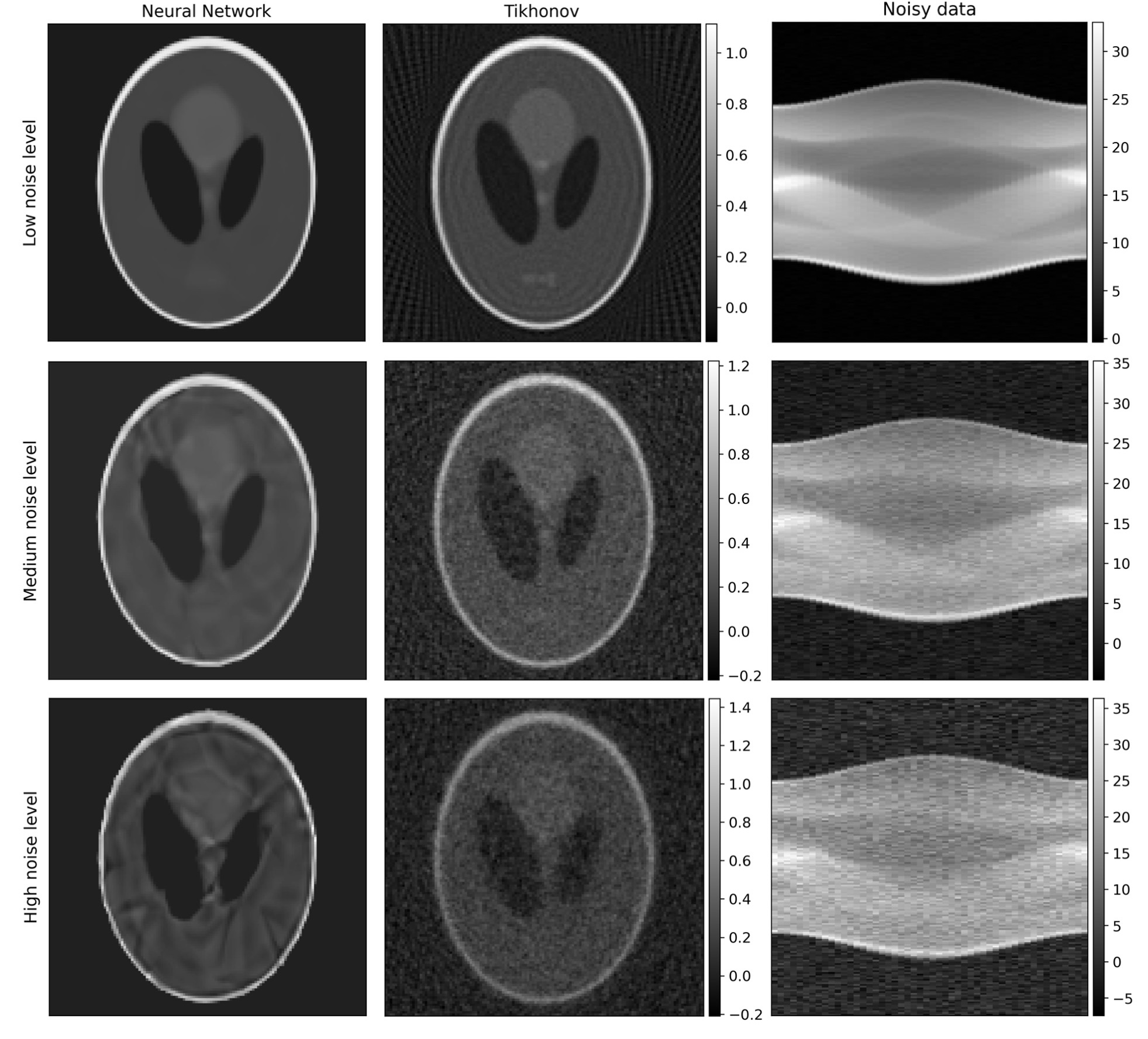}  
    \caption{Comparison of the NN reconstrction and the Tikhonov reconsturction for the first, 5 and the last value of noise. Low noise level, medium noise level, and high noise level correspond to SNR values of 42.60, 23.10, and 16.58, respectively.}\label{fig:example_reconstruction}
\end{figure}

We consider ten different values of $\delta$ in such a way that the resulting data Signal-to-Noise Ratio\footnote{The data SNR is defined as $20\log_{10}\left(\frac{\|{\bf y}\|}{\sqrt{M}\delta}\right)$ (see Chapter 3 of \cite{bertero1998introduction}).} (SNR) ranges between 16.6 dB and 42.6 dB.
For each noise level, we generate five random realizations of the data ${\bf y}^\delta$ and, for each noisy realization, we compute a set of regularized solutions $\{{\bf x}^\delta_{\alpha_i} \}_i$ and $\{{\bf x}^\delta_{\omega^\ast, \alpha_i}\}_i$ using the Tikhonov method and our proposed approach, respectively.
Then, we compute the reconstruction error with respect to the ground truth image using the $\ell^2$ norm, i.e., 
\begin{equation}
\mathrm{err}_{\rm Tik}(i, \delta, {\bf y}^\delta) = \| {\bf x} -  {\bf x}^\delta_{\alpha_i} \| ~,~ \mathrm{err}_{\rm NN}(i, \delta, {\bf y}^\delta) = \| {\bf x} -  {\bf x}^\delta_{\omega^\ast, \alpha_i} \| ~,
\end{equation}
and we select the reconstructions ${\bf x}^\delta_{\alpha_{i^\ast}}$ and ${\bf x}^\delta_{\omega^\ast, \alpha_{i^\ast}}$ that minimize such errors.
Figure \ref{fig:example_reconstruction} shows the optimal reconstructions provided by our proposed method and by the Tikhonov method for three different levels of data SRN.
With a slight amount of license, in the following we will denote by $\mathrm{err}_{\rm Tik}(\delta)$ and $\mathrm{err}_{\rm NN}(\delta)$ the average value over the five different noise realizations of the reconstruction errors corresponding to the optimal regularized solutions.

To compare the convergence properties of our NN-based regularization method with the Tikhonov ones, in Figure \ref{fig:error_vs_delta} we plot the quantities $\mathrm{err}_{\rm Tik}$ and $\mathrm{err}_{\rm NN}$ as a function of $\delta$.
The standard deviations of the reconstruction errors computed over the five different noise realizations are plotted as vertical error bars.
As a reference, we overlay a quantity proportional to $\delta^{2/3}$, which represents the optimal rate of convergence derived in section \ref{sec:neural_networks} and in \cite{engl1996regularization}.

\begin{figure}[h]
    \centering
    \includegraphics[width=0.9\textwidth]{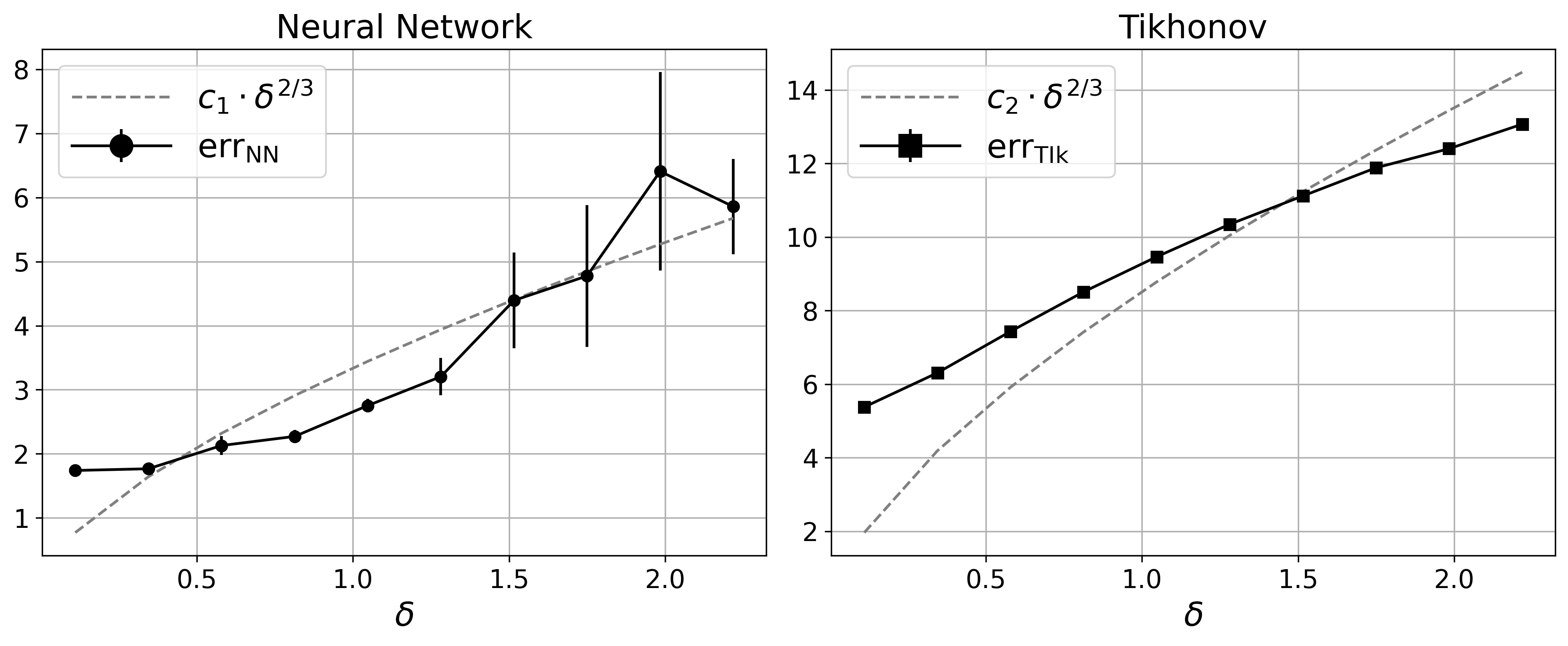}  
    \caption{Average reconstruction error over $5$ different noise realization for both our proposed method and the Tikhonov method, as a function of the noise level $\delta$ (left and right panel, respectively). The error bars represent the standard deviation across the same noise realizations. For comparison with optimal convergence rates, a quantity proportional to $\delta^{2/3}$ is plotted with a dashed line in both panels. The values $c_1$ and $c_2$ are determined by means the least square method.}
    \label{fig:error_vs_delta}
\end{figure}

\subsection{Analysis of the results}

Our proposed NN-based regularization technique outperforms the Tikhonov method, as demonstrated by the lower reconstruction errors achieved (see Figure \ref{fig:error_vs_delta}).
This is likely due to two main reasons.
First, the NN-based reconstructions have non-negative pixel values as they belong to the constraint set $C = \mathbb{R}^N_+$.
Second, the solutions obtained with our proposed method appear to better reproduce the piece-wise constant pixel values of the ground truth image, although no additional prior information is exploited by the NN-based method.
This is particularly visible in the reconstructions corresponding to the highest SNR value (top panels of Figure \ref{fig:example_reconstruction}), where the Tikhonov reconstruction shows evident ringing artifacts.
However, the NN-based regularization method provides an over-smoothed solution, in which the fine structures present in the bottom part of the Shepp-
Logan phantom are barely visible.
The ``smoothness'' property introduced by our proposed regularization method is also visible in the reconstructions corresponding to low and medium SNR values, especially in the parts of the image outside the phantom.
Finally, the NN-based technique reconstructs sharper edges compared to classic Tikhonov, as it is visible in areas of the ground truth object representing the skull.

Both our proposed regularization technique and the Tikhonov method show a convergence rate of the reconstruction error close to the optimal rate $\mathcal{O}(\delta^{2/3})$, as demonstrated in Figure \ref{fig:error_vs_delta}.
The discrepancy between the error rates derived in the numerical experiments and the theoretical ones is likely caused by the fact that theoretical results are demonstrated in an infinite-dimensional setting, whereas numerical experiments are performed upon discretization of infinite-dimensional problem.
Furthermore, the classical Tikhonov method, which relies on a closed-form solution, appears to be stable with respect to the presence of stochastic noise in the data, as demonstrated by the almost-negligible error bars plotted in the right panel of Figure \ref{fig:error_vs_delta}.
Differently, our proposed regularization method consists of an NN optimization, which is known to be a numerically unstable process, and is therefore sensitive to the specific noise realization affecting the data, especially for large noise levels (see the left panel of Figure \ref{fig:error_vs_delta}).

\section{Conclusions}
\label{sec:conclusion}
In this paper we presented the use of a NN to approximate the unknown object of an inverse problem inspired by the success of the neural radiance field strategy for natural image reconstruction \cite{mildenhall2021nerf}. 
First, we studied its asymptotic approximation properties when the weights of the network are regularized by the classical Tikhonov regularization method.
Then, we observed by numerical computation its non-asymptotic properties in a set of controlled experiments, where the proposed method provided promising results especially in the presence of higher amount of noise, by showing edge-preserving properties without having introduced an explicit constraint in this regard.
Although this aspect of the numerical results is encouraging, a deeper investigation assessing the performance of the method on objects with different morphologies has to be done in the next future.

\section*{Acknowledgments}
The authors thank the National Group of Scientific Computing (GNCS-INDAM) that supported this research.
PM is supported by the Swiss National Science Foundation in the framework of the project Robust Density Models for High Energy Physics and Solar Physics (rodem.ch), CRSII5\_193716.

\bibliographystyle{siamplain}

\end{document}